\documentclass[a4paper,12pt,oneside]{amsart}

\usepackage[applemac]{inputenc}
\usepackage[british]{babel}
\usepackage{fancyhdr}
\usepackage{geometry}
\usepackage{setspace}
\usepackage{bbm}
\usepackage{amssymb}

\newlength{\aufzleft}
\newenvironment{aufz}{\begin{list}{}{\setlength{\listparindent}{0pt}\setlength{\itemsep}{\topsep}\setlength{\labelwidth}{3.2ex}\setlength{\aufzleft}{\labelsep}\addtolength{\aufzleft}{\labelwidth}\setlength{\leftmargin}{\aufzleft}}}{\end{list}}
\newenvironment{aufz2}{\begin{list}{}{\setlength{\listparindent}{0pt}\setlength{\itemsep}{\topsep}\setlength{\labelwidth}{8ex}\setlength{\aufzleft}{\labelsep}\addtolength{\aufzleft}{\labelwidth}\setlength{\leftmargin}{\aufzleft}}}{\end{list}}

\newtheoremstyle{par}{1ex}{2ex}{\rm}{}{\bfseries}{}{0.8em}{\thmnumber{(#2)}}
\newtheoremstyle{thm}{1ex}{2ex}{\itshape}{}{\bfseries}{}{0.9em}{\thmnumber{(#2)}\thmname{ #1}\thmnote{ (#3)}}
\newtheoremstyle{ex}{1ex}{2ex}{\rm}{}{\bfseries}{}{0.8em}{\thmnumber{(#2)}\thmname{ #1}}

\theoremstyle{par}
\newtheorem{no}{}[section]

\theoremstyle{thm}
\newtheorem{lemma}[no]{Lemma}
\newtheorem{prop}[no]{Proposition}
\newtheorem{cor}[no]{Corollary}

\theoremstyle{ex}
\newtheorem{exas}[no]{Examples}
\newtheorem{qus}[no]{Questions}

\newcommand{\N}{\mathbbm{N}}
\newcommand{\Z}{\mathbbm{Z}}

\newcommand{\hm}[3]{{\rm Hom}_{#1}(#2,#3)}
\newcommand{\ia}{\mathfrak{a}}
\newcommand{\ib}{\mathfrak{b}}
\newcommand{\ic}{\mathfrak{c}}
\newcommand{\ip}{\mathfrak{p}}
\newcommand{\iq}{\mathfrak{q}}
\newcommand{\im}{\mathfrak{m}}
\newcommand{\inn}{\mathfrak{n}}
\newcommand{\catmod}{{\sf Mod}}
\newcommand{\assf}{\ass^{\rm f}}
\newcommand{\Id}{{\rm Id}}
\newcommand{\ilim}{\varinjlim}

\DeclareMathOperator{\ass}{Ass}
\DeclareMathOperator{\var}{Var}
\DeclareMathOperator{\nzd}{NZD}
\DeclareMathOperator{\zd}{ZD}
\DeclareMathOperator{\spec}{Spec}
\DeclareMathOperator{\supp}{Supp}
\DeclareMathOperator{\ke}{Ker}
\DeclareMathOperator{\nil}{Nil}

\begin{document}

\newcommand{\leadingzero}[1]{\ifnum #1<10 0\the#1\else\the#1\fi}
\renewcommand{\today}{\the\year\leadingzero{\month}\leadingzero{\day}}

\title{Assassins and torsion functors}
\author{Fred Rohrer}
\address{Grosse Grof 9, 9470 Buchs, Switzerland}
\email{fredrohrer@math.ch}
\subjclass[2010]{Primary 13C12; Secondary 13D30, 13D45}
\keywords{Torsion functor, assassin, weak assassin, non-noetherian ring, well-centered torsion theory, weakly proregular ideal}

\begin{abstract}
Let $R$ be a ring, let $\ia\subseteq R$ be an ideal, and let $M$ be an $R$-module. Let $\Gamma_\ia$ denote the $\ia$-torsion functor. Conditions are given for the (weakly) associated primes of $\Gamma_\ia(M)$ to be the (weakly) associated primes of $M$ containing $\ia$, and for the (weakly) associated primes of $M/\Gamma_\ia(M)$ to be the (weakly) associated primes of $M$ not containing $\ia$.
\end{abstract}

\maketitle\thispagestyle{fancy}


\section{Introduction}

Let $R$ be a ring\footnote{Throughout what follows, rings are understood to be commutative. In general, notation and terminology follow Bourbaki's \textit{\'El\'ements de math\'e\-matique.}}, let $\ia\subseteq R$ be an ideal, and let $M$ be an $R$-module. We denote by $\Gamma_\ia(M)=\bigcup_{n\in\N}(0:_M\ia^n)$ the $\ia$-torsion submodule of $M$. The $\ia$-torsion functor $\Gamma_\ia$, and especially its right derived cohomological functor, i.e., local cohomology with respect to $\ia$, are important tools in commutative algebra and algebraic geometry. They are mostly studied in case the ring $R$ is noetherian. If $R$ is not noetherian, then these functors may behave differently to what is known from the noetherian case. For example, the functors $\Gamma_\ia$ and $\Gamma_{\sqrt{\ia}}$ need not be equal (\cite{r}), and $\Gamma_\ia(I)$ need not be injective for an injective $R$-module $I$ (\cite{qr}). So, when trying to extend the theory of local cohomology to not necessarily noetherian rings, one has to be careful, and some basic results on local cohomology need additional hypotheses in order to still hold. Another approach, for supporting ideals of finite type, is via the notion of weak proregularity as explained in \cite{lipman} (cf. \ref{wr}).

In this article, we study the interplay between torsion functors and assassins. Let $\ass_R(M)$ denote the assassin of $M$, i.e., the set of prime ideals of $R$ of the form $(0:_Rx)$ for some $x\in M$. It is well-known that noetherianness of $R$ implies the following nice relations between the assassins of $\Gamma_\ia(M)$, $M$ and $M/\Gamma_\ia(M)$:
\begin{aufz2}
\item[(A)] $\ass_R(\Gamma_\ia(M))=\ass_R(M)\cap\var(\ia)$;
\item[(B)] $\ass_R(M/\Gamma_\ia(M))=\ass_R(M)\setminus\var(\ia)$.
\end{aufz2}
Our first goal is to investigate these relations in case $R$ is non-noetherian. Our main results are as follows. Statement (A) and the inclusion ``$\supseteq$'' in (B) hold without any hypothesis (\ref{b10}). Statement (B) need not hold in general (\ref{b70}, \ref{b80}), but it does so in the following cases: i) $M$ is noetherian (\ref{b50}); ii) $\ia$ is principal and idempotent, and $M=R$ (\ref{c10}); iii) $\ia$ and all associated primes of $M/\Gamma_\ia(M)$ are of finite type (\ref{g40}).\smallskip

\chead{\footnotesize Assassins and torsion functors}

In non-noetherian situations, assassins are often not so well-behaved, in contrast to weak assassins. The weak assassin of $M$, denoted by $\assf_R(M)$, is the set of prime ideals of $R$ that are minimal primes of ideals of the form $(0:_Rx)$ for some $x\in M$. Our second goal is to investigate the analogues to (A) and (B) for weak assassins:
\begin{aufz2}
\item[(A$'$)] $\assf_R(\Gamma_\ia(M))=\assf_R(M)\cap\var(\ia)$;
\item[(B$'$)] $\assf_R(M/\Gamma_\ia(M))=\assf_R(M)\setminus\var(\ia)$.
\end{aufz2}
Here, our main results are as follows. The inclusion ``$\subseteq$'' in (A$'$) and the inclusion ``$\supseteq$'' in (B$'$) hold without any hypothesis (\ref{b10}), and (B$'$) implies (A$'$) (\ref{b30}). Neither (A$'$) nor (B$'$) holds in general (\ref{b70}, \ref{b80}). Statement (A$'$) holds if $\ia$ is of finite type (\ref{f20}). Statement (B$'$) holds if $\ia$ is principal (\ref{f40}), and if $\ia$ has a finite generating family $(a_i)_{i=1}^n$ such that $(a_i)_{i=1}^m$ is weakly proregular for every $m\in[1,n-1]$ and that $a_i$ is weakly proregular for every $i\in[2,n]$ (\ref{f60}).\smallskip

If $\Gamma_\ia$ is a radical (e.g., if $\ia$ is of finite type), then it defines a torsion theory. This torsion theory is well-centered in the sense of Cahen (\cite{cahen1}) if and only if (A$'$) holds for any $R$-module $M$ (\ref{d35}).\smallskip

Finally, let us note that there is another way to sweeten a non-noetherian situation. Namely, instead of $\Gamma_\ia$ we can consider the left exact radical $\overline{\Gamma}_\ia$ defined by $\overline{\Gamma}_\ia(M)=\{x\in M\mid\supp_R(\langle x\rangle_R)\subseteq\var(\ia)\}$. It coincides with $\Gamma_\ia$ if $\ia$ is of finite type, but not in general. We will pursue the study of assassins and weak assassins of $\overline{\Gamma}_\ia(M)$ and $M/\overline{\Gamma}_\ia(M)$ in a subsequent work.


\section{Preliminaries}

\noindent\textit{Throughout this section, let $R$ be a ring, let $\ia\subseteq R$ be an ideal, and let $M$ be an $R$-module.}

\begin{no}
We denote by $\nil(R)$ the nilradical of $R$, by $\var(\ia)$ the set of prime ideals of $R$ containing $\ia$, by $\min(\ia)$ the set of $\subseteq$-minimal elements of $\var(\ia)$, by $\nzd_R(M)$ and $\zd_R(M)$ the sets of non-zerodivisors and zerodivisors on $M$, resp., and by $\supp_R(M)$ the support of $M$.
\end{no}

\begin{no}
A prime ideal $\ip\subseteq R$ is said to be \textit{associated to $M$} if there exists $x\in M$ with $\ip=(0:_Rx)$; it is said to be \textit{weakly associated to $M$} if there exists $x\in M$ with $\ip\in\min(0:_Rx)$. The sets $\ass_R(M)$ and $\assf_R(M)$ of associated and weakly associated primes of $M$ are called \textit{the assassin of $M$} and \textit{the weak assassin of $M$,} resp.\smallskip
\end{no}

If not indicated otherwise, the following facts about assassins and weak assassins are proven in \cite[IV.1.1; IV.1 Exercise 17]{ac} and \cite[00L9; 0546]{stacks}.

\begin{no}\label{a20}
A) We have $\{\ip\in\assf_R(M)\mid\ip\text{ is of finite type}\}\subseteq\ass_R(M)\subseteq\assf_R(M)$; if $R$ or $M$ is noetherian, then $\ass_R(M)=\assf_R(M)$.\smallskip

B) If $\ip\in\spec(R)$, then $\ass_R(R/\ip)=\assf_R(R/\ip)=\{\ip\}$.\smallskip

C) We have $M=0$ if and only if $\assf_R(M)=\emptyset$.\smallskip

D) If $0\rightarrow L\rightarrow M\rightarrow N\rightarrow 0$ is a short exact sequence of $R$-modules, then \[\ass_R(L)\subseteq\ass_R(M)\subseteq\ass_R(L)\cup\ass_R(N)\] and \[\assf_R(L)\subseteq\assf_R(M)\subseteq\assf_R(L)\cup\assf_R(N).\]

E) If $(M_i)_{i\in I}$ is a filtering inductive system of $R$-modules, then \[\{\ip\in\ass_R(\ilim_{i\in I}M_i)\mid\ip\text{ is of finite type}\}\subseteq\bigcup_{i\in I}\ass_R(M_i),\] as follows from the proof of \cite[2.10]{coupek}.\smallskip

F) Let $S\subseteq R$ be a subset, and let $\eta\colon M\rightarrow S^{-1}M$ denote the canonical morphism. Then, $\assf_R(\ke(\eta))=\{\ip\in\assf_R(M)\mid\ip\cap S=\emptyset\}$, and there is a bijection \[\{\ip\in\assf_R(M)\mid\ip\cap S=\emptyset\}\overset{\cong}\longrightarrow\assf_{S^{-1}R}(S^{-1}M),\;\ip\mapsto S^{-1}\ip.\]

G) We have $\supp_R(M)=\{\ip\in\spec(R)\mid\exists\iq\in\assf_R(M):\iq\subseteq\ip\}$. If $N\subseteq M$ is a sub-$R$-module, then $\supp_R(M)=\supp_R(N)\cup\supp_R(M/N)$ (\cite[II.4.4 Proposition 16]{ac}).\smallskip

H) Suppose $M$ is of finite type, and let $N$ be a further $R$-module. In \cite[IV.1.4 Proposition 10]{ac}, it is shown that if $R$ is noetherian, then \[\ass_R(\hm{R}{M}{N})=\ass_R(N)\cap\supp_R(M).\] A careful analysis of the proof reveals that the noetherian hypothesis is superfluous, and moreover that $\assf_R(\hm{R}{M}{N})\subseteq\assf_R(N)\cap\supp_R(M)$. This last inclusion need not be an equality, as can be seen from \cite[Example 4]{yassemi}.
\end{no}

If not indicated otherwise, the following facts about torsion functors are proven in \cite[Chapters 1--2]{bs} (where one has to check that the hypothesis of a noetherian ring is not used) and \cite[Chapter 1]{bfr}.

\begin{no}\label{tor}
A) The $\ia$-torsion functor, denoted by $\Gamma_\ia$, is the subfunctor of $\Id_{\catmod(R)}$ with $\Gamma_\ia(M)=\bigcup_{n\in\N}(0:_M\ia^n)$ for every $R$-module $M$. It is left exact.\smallskip

B) A subfunctor $F$ of $\Id_{\catmod(R)}$ is called a \textit{radical} if $F(M/F(M))=0$ for every $R$-module $M$. The functor $\Gamma_\ia$ need not be a radical, but it is so if $\ia$ is of finite type. Moreover, if $M$ is noetherian, then $\Gamma_\ia(M/\Gamma_\ia(M))=0$.\smallskip

C) Suppose that $M$ is noetherian. Then, there exists $n\in\N$ with $\ia^n\Gamma_\ia(M)=0$. If moreover $\Gamma_\ia(M)=0$, then $\nzd_R(M)\cap\ia\neq\emptyset$.\smallskip

D) Let $S\subseteq R$ be a subset. The canonical monomorphism of $S^{-1}R$-modules $S^{-1}\Gamma_\ia(M)\rightarrowtail\Gamma_{S^{-1}\ia}(S^{-1}M)$ need not be an isomorphism, but it is so if $\ia$ is of finite type.\smallskip

E) If $\Gamma_\ia$ is a radical and $\hm{R}{M}{N}=0$ for every $\ia$-torsionfree $R$-module $N$, then $\Gamma_\ia(M)=M$. (Indeed, we consider the short exact sequence \[0\longrightarrow\Gamma_\ia(M)\overset{h}\longrightarrow M\overset{l}\longrightarrow M/\Gamma_\ia(M)\longrightarrow 0.\] Then, $\Gamma_\ia(M/\Gamma_\ia(M))=0$, hence $l=0$, thus $h$ is an isomorphism, and therefore $\Gamma_\ia(M)=M$.)\smallskip

F) There is a canonical monomorphism of $R$-modules \[M/\Gamma_\ia(M)\rightarrowtail\ilim_{n\in\N}\hm{R}{\ia^n}{M}.\]
\end{no}

We recall the notion of weak proregularity of an ideal and refer the reader to \cite{lipman}, \cite{yekutieli} and \cite{schenzel} for details.

\begin{no}\label{wr}
A) Suppose $\ia$ is of finite type. Let $\mathbf{a}=(a_i)_{i=1}^n$ be a finite generating family of $\ia$. The right derived cohomological functor of $\Gamma_\ia$ is denoted by $(H_\ia^i)_{i\in\Z}$ and called \textit{local cohomology with respect to $\ia$.} \v{C}ech cohomology with respect to $\mathbf{a}$ yields a further cohomological functor, denoted by $(\check{H}^i(\mathbf{a},\bullet))_{i\in\Z}$. There is a canonical isomorphism $\Gamma_\ia(\bullet)\cong\check{H}^0(\mathbf{a},\bullet)$ that can be canonically extended to a morphism of $\delta$-functors $\gamma_{\mathbf{a}}\colon(H_\ia^i(\bullet))_{i\in\Z}\rightarrow(\check{H}^i(\mathbf{a},\bullet))_{i\in\Z}$. The sequence $\mathbf{a}$ is called \textit{weakly proregular} if $\gamma_{\mathbf{a}}$ is an isomorphism. The ideal $\ia$ is called \textit{weakly proregular} if it has a weakly proregular finite generating family, and this holds if and only if every finite generating family of $\ia$ is weakly proregular.\smallskip

B) If $\ia$ is weakly proregular and $i\in\N^*$, then $H_\ia^i\circ\Gamma_\ia=0$, as follows immediately from the definitions of weak proregularity and \v{C}ech cohomology.
\end{no}


\section{Basic relations, first examples, and counterexamples}

\noindent\textit{Throughout this section, let $R$ be a ring, let $\ia\subseteq R$ be an ideal, and let $M$ be an $R$-module.}

\begin{prop}\label{b10}
a) $\ass_R(\Gamma_\ia(M))=\ass_R(M)\cap\var(\ia)$.

b) $\assf_R(\Gamma_\ia(M))\subseteq\assf_R(M)\cap\var(\ia)$.

c) $\ass_R(M/\Gamma_\ia(M))\supseteq\ass_R(M)\setminus\var(\ia)$.

d) $\assf_R(M/\Gamma_\ia(M))\supseteq\assf_R(M)\setminus\var(\ia)$.
\end{prop}

\begin{proof}
We have $\ass_R(\Gamma_\ia(M))\subseteq\ass_R(M)$ and $\assf_R(\Gamma_\ia(M))\subseteq\assf_R(M)$ by \ref{a20} D). For $\ip\in\assf_R(\Gamma_\ia(M))$ there exists $x\in\Gamma_\ia(M)$ with $(0:_Rx)\subseteq\ip$, hence there exists $n\in\N$ with $\ia^nx=0$, and it follows that $\ia^n\subseteq(0:_Rx)\subseteq\ip$, thus $\ia\subseteq\ip$. By \ref{a20} A) this shows \[\ass_R(\Gamma_\ia(M))\subseteq\assf_R(\Gamma_\ia(M))\subseteq\var(\ia).\] So, we have proven b) and the inclusion ``$\subseteq$'' in a).

For $\ip\in\ass_R(M)\cap\var(\ia)$ there exists $x\in M$ with $\ia\subseteq\ip=(0:_Rx)$, implying $\ia x=0$, hence $x\in\Gamma_\ia(M)$, and thus $\ip\in\ass_R(\Gamma_\ia(M))$. This proves the inclusion ``$\supseteq$'' in a).

By \ref{a20} D) we have \[\ass_R(M)\subseteq\ass_R(\Gamma_\ia(M))\cup\ass_R(M/\Gamma_\ia(M))\] and \[\assf_R(M)\subseteq\assf_R(\Gamma_\ia(M))\cup\assf_R(M/\Gamma_\ia(M)).\] Thus, c) and d) follow immediately from a) and b).
\end{proof}

\begin{no}\label{b20}
The $R$-module $M$ is called \textit{$\ia$-fair} if \[\ass_R(M/\Gamma_\ia(M))=\ass_R(M)\setminus\var(\ia);\] it is called \textit{weakly $\ia$-quasifair} if \[\assf_R(\Gamma_\ia(M))=\assf_R(M)\cap\var(\ia),\] and \textit{weakly $\ia$-fair} if \[\assf_R(M/\Gamma_\ia(M))=\assf_R(M)\setminus\var(\ia).\]
\end{no}

\begin{prop}\label{b30}
If $M$ is weakly $\ia$-fair, then it is weakly $\ia$-quasifair.
\end{prop}

\begin{proof}
By \ref{a20} D) we have \[\assf_R(M)\cap\var(\ia)\subseteq\assf_R(\Gamma_\ia(M))\cup(\assf_R(M/\Gamma_\ia(M))\cap\var(\ia))=\]\[\assf_R(\Gamma_\ia(M))\cup((\assf_R(M)\setminus\var(\ia))\cap\var(\ia))=\assf_R(\Gamma_\ia(M)).\] So, the claim follows from \ref{b10} b).
\end{proof}

\begin{exas}\label{b40}
A) Every $R$-module is $R$-fair and weakly $R$-fair.\smallskip

B) If $\ia$ is nilpotent, then every $R$-module is $\ia$-fair and weakly $\ia$-fair.\smallskip

C) Every $\ia$-torsion $R$-module is $\ia$-fair and weakly $\ia$-fair.\smallskip

D) If $\ip\in\spec(R)$, then the $R$-module $R/\ip$ is $\ia$-fair and weakly $\ia$-fair, as follows on use of \ref{a20} B). In particular, if $R$ is integral, then the $R$-module $R$ is $\ia$-fair and weakly $\ia$-fair.\smallskip

E) Suppose $\ia\subseteq\nil(R)$. Then, $M$ is $\ia$-fair if and only if $\ass_R(M/\Gamma_\ia(M))=\emptyset$, and this is fulfilled if $\Gamma_\ia$ is a radical. Furthermore, $M$ is weakly $\ia$-fair if and only if $\assf_R(M/\Gamma_\ia(M))=\emptyset$, and, by \ref{a20} C), this holds if and only if $M$ is an $\ia$-torsion $R$-module.
\end{exas}

\begin{prop}\label{b45}
If every monogeneous $R$-module is weakly $\ia$-quasifair, $\ia$-fair, or weakly $\ia$-fair, resp., then so is every $R$-module.
\end{prop}

\begin{proof}
Let $M$ be an $R$-module. Suppose that every monogeneous $R$-module is $\ia$-fair or weakly $\ia$-fair, resp. Let $\ip\in\ass_R(M/\Gamma_\ia(M))$ or $\ip\in\assf_R(M/\Gamma_\ia(M))$, resp. There exists $x\in M$ with $\ip=(0:_R\overline{x})$ or $\ip\in\min(0:_R\overline{x})$, resp., where $\overline{x}$ denotes the canonical image of $x$ in $M/\Gamma_\ia(M)$. Then, $\langle x\rangle_R/\Gamma_\ia(\langle x\rangle_R)\cong\langle\overline{x}\rangle_R\subseteq M/\Gamma_\ia(M)$, and thus \[\ip\in\ass_R(\langle\overline{x}\rangle_R)=\ass_R(\langle x\rangle)\setminus\var(\ia)\subseteq\ass_R(M)\setminus\var(\ia)\] or \[\ip\in\assf_R(\langle\overline{x}\rangle_R)=\assf_R(\langle x\rangle)\setminus\var(\ia)\subseteq\assf_R(M)\setminus\var(\ia),\] resp. by \ref{a20} D). Now, \ref{b10} c), d) implies that $M$ is $\ia$-fair or weakly $\ia$-fair, resp.

Suppose that every monogeneous $R$-module is weakly $\ia$-quasifair. Let $\ip\in\assf_R(M)\cap\var(\ia)$. There exists $x\in M$ with $\ip\in\min(0:_Rx)$. It follows that \[\ip\in\assf_R(\langle x\rangle_R)\cap\var(\ia)=\assf_R(\Gamma_\ia(\langle x\rangle_R))\subseteq\assf_R(\Gamma_\ia(M))\] by \ref{tor} A) and \ref{a20} D), and so \ref{b10} b) implies that $M$ is weakly $\ia$-quasifair.
\end{proof}

\begin{prop}\label{b50}
If $R$ or $M$ is noetherian, then $M$ is $\ia$-fair and weakly $\ia$-fair.
\end{prop}

\begin{proof}
By \ref{a20} A) it suffices to show that $M$ is $\ia$-fair, and by \ref{b45} it suffices to consider the case that $M$ is noetherian. Let $\ip\in\ass_R(M/\Gamma_\ia(M))$. Since $M$ is noetherian, we have $\Gamma_\ia(M/\Gamma_\ia(M))=0$ by \ref{tor} B), so since $M/\Gamma_\ia(M)$ is noetherian, too, there exists $x\in\nzd_R(M/\Gamma_\ia(M))\cap\ia$ by \ref{tor} C). As $\ip\subseteq\bigcup\ass_R(M/\Gamma_\ia(M))\subseteq\zd_R(M/\Gamma_\ia(M))$, it follows that $x\in\ia\setminus\ip$, and therefore $\ip\notin\var(\ia)$. Furthermore, there exists $v\in M$ such that, setting $\overline{v}=v+\Gamma_\ia(M)\in M/\Gamma_\ia(M)$, we have $\ip=(0:_R\overline{v})$, hence $\ip\overline{v}=0$, and thus $\ip v\subseteq\Gamma_\ia(M)$. Since $M$ is noetherian, there exists $n\in\N$ with $\ia^n\Gamma_\ia(M)=0$ by \ref{tor} C), implying $\ip x^nv=x^n\ip v\subseteq x^n\Gamma_\ia(M)=0$, and thus $\ip\subseteq(0:_Rx^nv)$. If $a\in(0:_Rx^nv)$, then $(ax^n)v=a(x^nv)=0\in\Gamma_\ia(M)$, hence $ax^n\overline{v}=0$, thus $ax^n\in(0:_R\overline{v})=\ip$, and as $x\notin\ip$ it follows that $a\in\ip$. Therefore, $(0:_Rx^nv)\subseteq\ip$, hence $\ip=(0:_Rx^nv)$, and thus $\ip\in\ass_R(M)$. The claim follows now from \ref{b10} c).
\end{proof}

\begin{prop}\label{b60}
Let $R$ be $0$-dimensional and local with maximal ideal $\im$.

a) The $R$-module $R$ is weakly $\im$-quasifair if and only if $\Gamma_\im(R)\neq 0$.

b) The $R$-module $R$ is weakly $\im$-fair if and only if $\im$ is nilpotent.

c) If $\Gamma_\im(R)=\im$, then the $R$-module $R$ is weakly $\im$-quasifair, but neither $\im$-fair nor weakly $\im$-fair.

d) If $\Gamma_\im(R)=0$, then the $R$-module $R$ is $\im$-fair, but not weakly $\im$-quasifair.
\end{prop}

\begin{proof}
a) The $R$-module $R$ is weakly $\im$-quasifair if and only if $\assf_R(\Gamma_\im(R))=\{\im\}$, hence if and only if $\assf_R(\Gamma_\im(R))\neq\emptyset$, and thus, by \ref{a20} C), if and only if $\Gamma_\im(R)\neq 0$.

b) The $R$-module $R$ is weakly $\im$-fair if and only if $\assf_R(R/\Gamma_\im(R))=\emptyset$, thus, by \ref{a20} C), if and only $R/\Gamma_\im(R)=0$, and hence if and only if $\im$ is nilpotent.

c) The $R$-module $R$ is weakly $\im$-quasifair by a). As $\ass_R(R)\setminus\var(\im)=\assf_R(R)\setminus\var(\im)=\emptyset$, we get, on use of \ref{a20} A), \[\{\im\}\subseteq\ass_R(R/\im)=\ass_R(R/\Gamma_\im(R))\subseteq\assf_R(R/\Gamma_\im(R))\subseteq\{\im\},\] and therefore the remaining claims.

d) As $\Gamma_\im(R)=0$, we have $\im\notin\ass_R(R)\subseteq\var(\im)$. On use of \ref{b10} a) we thus get \[\ass_R(R)\setminus\var(\im)=\emptyset=\ass_R(\Gamma_\im(R))=\]\[\ass_R(R)\cap\var(\im)=\ass_R(R)=\ass_R(R/\Gamma_\im(R)).\] Hence, the $R$-module $R$ is $\im$-fair. By a), it is not weakly $\ia$-quasifair.
\end{proof}

\begin{prop}\label{b70}
There exist a ring $R$, an ideal $\ia\subseteq R$, and an $R$-module $M$ such that $M$ is weakly $\ia$-quasifair, but neither $\ia$-fair nor weakly $\ia$-fair.\footnote{This shows that the converse of \ref{b30} need not hold.}
\end{prop}

\begin{proof}
Let $K$ be a field and let \[R=K[(X_i)_{i\in\N}]/\langle\{X_iX_j\mid i,j\in\N,i\neq j\}\cup\{X_i^{i+1}\mid i\in\N\}\rangle_{K[(X_i)_{i\in\N}]}.\] For $i\in\N$ we denote by $Y_i$ the image of $X_i$ in $R$. Let $\im=\langle Y_i\mid i\in\N\rangle_R$. Then, $R$ is a $0$-dimensional local ring with maximal ideal $\im$. If $n\in\N$, then $\im^n=\langle Y_i^n\mid i\geq n\rangle_R$, hence $(0:_R\im^n)=\langle Y_i^{\max\{1,i-n+1\}}\mid i\in\N\rangle_R$, and thus $\Gamma_{\im}(R)=\im$. Now, setting $\ia=\im$ and $M= R$, the claim follows from \ref{b60}.
\end{proof}

\begin{prop}\label{b80}
There exist a ring $R$, an ideal $\ia\subseteq R$, and an $R$-module $M$ such that $M$ is $\ia$-fair, but not weakly $\ia$-quasifair.
\end{prop}

\begin{proof}
Let $K$ be a field, let $Q$ denote the additive monoid of positive rational numbers, let $R=K[Q]$ denote the algebra of $Q$ over $K$, and let $\{e_{\alpha}\mid\alpha\in Q\}$ denote its canonical basis. Then, $\im=\langle e_{\alpha}\mid\alpha>0\rangle_R$ is a maximal ideal. We consider $S=R_{\im}$ and $\inn=\im_{\im}$. Then, $S$ is a $1$-dimensional local valuation ring with idempotent maximal ideal $\inn$ (\cite[2.2]{qr}). Let $\ia=\langle\frac{e_1}{1}\rangle_S$, let $T=S/\ia$, and let $\ip=\inn/\ia$. Then, $T$ is a $0$-dimensional local ring with idempotent maximal ideal $\ip$.

We will show now that $\Gamma_{\ip}(T)=0$, and then \ref{b60} d) will yield the claim. Since $\ip$ is idempotent, we have to show that if $f\in T$ with $\ip f=0$, then $f=0$. By construction of $T$ and $S$ we have to show that if $f\in R$ with $\inn\frac{f}{1}\subseteq\ia$, then $\frac{f}{1}\in\ia$. So, let $f\in R$ with $\inn\frac{f}{1}\subseteq\ia$. We assume that $\frac{f}{1}\notin\ia$ and will derive a contradiction. There exist $l\in\N$, a family $(f_i)_{i=0}^l$ in $K\setminus 0$, and a strictly increasing family $(\beta_i)_{i=0}^l$ in $Q$ with $f=\sum_{i=0}^lf_ie_{\beta_i}$. If $\beta_0\geq 1$, then we get the contradiction $\frac{f}{1}\in\ia$. So, $\beta_0<1$, and hence there exists $\alpha\in Q$ with $\alpha>0$ and $\beta_0+\alpha<1$. As $\inn\frac{f}{1}\subseteq\ia$ it follows that $\frac{e_\alpha}{1}\frac{f}{1}\in\ia$, and thus there exist $g\in R\setminus 0$ and $t\in R\setminus\im$ with $\frac{fe_\alpha}{1}=\frac{ge_1}{t}\in S$. As $R$ is integral this implies \[tfe_\alpha=ge_1.\leqno{(*)}\] There exist $k\in\N$, a family $(t_i)_{i=0}^k$ in $K\setminus 0$, and a strictly increasing family $(\alpha_i)_{i=1}^k$ in $Q\setminus 0$ with $t=t_0+\sum_{i=1}^kt_ie_{\alpha_i}$. Furthermore, there exist $m\in\N$, a family $(g_i)_{i=0}^m$ in $K\setminus 0$, and a strictly increasing family $(\gamma_i)_{i=0}^m$ in $Q$ with $g=\sum_{i=0}^mg_ie_{\gamma_i}$. Thus, $(*)$ takes the form \[\sum_{j=0}^lt_0f_je_{\alpha+\beta_j}+\sum_{i=1}^k\sum_{j=0}^lt_if_je_{\alpha+\alpha_i+\beta_j}=\sum_{i=0}^mg_ie_{\gamma_i+1}.\leqno{(**)}\] Both sides of $(**)$ being nonzero, there exists the smallest $\delta\in Q$ such that $e_\delta$ occurs on both sides. As $\alpha_i>0$ for every $i\in[1,k]$, the left side yields $\delta=\alpha+\beta_0$, while the right side yields $\delta=\gamma_0+1$. So, we get the contradiction $1>\alpha+\beta_0=\gamma_0+1\geq 1$, and thus $\Gamma_{\ip}(T)=0$ as claimed.
\end{proof}

\begin{qus}\label{b90}
We saw in \ref{b30} that a weakly $\ia$-fair $R$-module is weakly $\ia$-quasifair. On the other hand, \ref{b70} and \ref{b80} show that a weakly $\ia$-quasifair $R$-module need not be $\ia$-fair or weakly $\ia$-fair, and that an $\ia$-fair $R$-module need not be weakly $\ia$-quasifair or weakly $\ia$-fair. Thus, in this general setting we are left with the following questions:
\begin{aufz}
\item[$(*)$] Is a weakly $\ia$-fair $R$-module necessarily $\ia$-fair?
\item[$(**)$] Is a weakly $\ia$-quasifair and $\ia$-fair $R$-module necessarily weakly $\ia$-fair?
\end{aufz}
By combinatorial observations one can show that if $R$ is a $0$-dimensional local ring or a $1$-dimensional local ring with a single minimal prime ideal that is of finite type, then every weakly $\ia$-fair $R$-module is $\ia$-fair. But we were not able to answer the above questions in general.
\end{qus}


\section{Results in the radical case}

\noindent\textit{Throughout this section, let $R$ be a ring, and let $\ia\subseteq R$ be an ideal.}\smallskip

\begin{prop}\label{d60}
If $\Gamma_{\ia}$ is a radical and $M$ is an $R$-module, then \[\ass_R(M/\Gamma_{\ia}(M))\cap\var(\ia)=\emptyset.\]
\end{prop}

\begin{proof}
Applying \ref{b10} a) to the $R$-module $M/\Gamma_\ia(M)$ yields \[\emptyset=\ass_R(0)=\ass_R(\Gamma_\ia(M/\Gamma_\ia(M)))=\ass_R(M/\Gamma_\ia(M))\cap\var(\ia),\] and thus our claim.
\end{proof}

\begin{prop}\label{d10}
Let $M$ be an $R$-module. Then:

a) $\assf_R(M)\cap\var(\ia)=\emptyset\Rightarrow\Gamma_\ia(M)=0\Rightarrow\ass_R(M)\cap\var(\ia)=\emptyset$.

b) $\Gamma_\ia(M)=M\Rightarrow\assf_R(M)\subseteq\var(\ia)\Rightarrow\ass_R(M)\subseteq\var(\ia)$.
\end{prop}

\begin{proof}
a) Suppose $\assf_R(M)\cap\var(\ia)=\emptyset$, and assume there exists $x\in\Gamma_\ia(M)\setminus 0$. Then, there exists $\ip\in\min(0:_Rx)$, and so we have $\ip\in\assf_R(M)$. Moreover, we have the situation \[\Gamma_\ia(M)\supseteq\langle x\rangle_R\overset{\cong}\longleftarrow R/(0:_Rx)\twoheadrightarrow R/\ip,\] implying $\Gamma_\ia(R/\ip)=R/\ip$, and therefore the contradiction $\ia\subseteq\ip$. This proves the first implication. The second one follows from \ref{b10} a).

b) (cf. \cite[1.14]{qr}) If $\Gamma_\ia(M)=M$ and $\ip\in\assf_R(M)$, then there exists $x\in M$ such that $\ip\in\min(0:_Rx)$, and there exists $n\in\N$ with $\ia^nx=0$. It follows that $\ia^n\subseteq(0:_Rx)\subseteq\ip$, hence $\ia\subseteq\ip$. This proves the first implication. The second one holds by \ref{a20} A).
\end{proof}

\begin{no}\label{d20}
A) The ideal $\ia$ is called \textit{centered} if the first implication in \ref{d10} a) is an equivalence for every $R$-module $M$, i.e., if \[\Gamma_\ia(M)=0\Leftrightarrow\assf_R(M)\cap\var(\ia)=\emptyset\] for every $R$-module $M$; it is called \textit{half-centered} if the first implication in \ref{d10} b) is an equivalence for every $R$-module $M$, i.e., if \[\Gamma_\ia(M)=M\Leftrightarrow\assf_R(M)\subseteq\var(\ia)\] for every $R$-module $M$. Finally, $\ia$ is called \textit{well-centered} if it is centered and half-centered.\smallskip

B) Suppose $\Gamma_\ia$ is a radical (\ref{tor} B)). Denoting by $T_\ia$ and $F_\ia$ the classes of $\ia$-torsion and $\ia$-torsionfree $R$-modules, resp., we get a torsion theory $(T_\ia,F_\ia)$. It is readily checked now that $(T_\ia,F_\ia)$ is half-centered or well-centered, resp., in the sense of \cite{cahen1} if and only if the ideal $\ia$ is so.\smallskip

C) (cf. \cite[1.14; 1.21 B)]{qr}) An arbitrary ideal need not be half-centered, but ideals of finite type are so. Indeed, if $R$ is a $0$-dimensional local ring whose maximal ideal $\im$ is not nilpotent (e.g., the ring in \ref{b70}), then $R$ is not an $\im$-torsion module, but $\assf_R(R)\subseteq\spec(R)=\var(\im)$, and therefore $\im$ is not half-centered. On the other hand, if $\ia$ is of finite type and $\assf_R(M)\subseteq\var(\ia)$, then for $x\in M$ we have $\ia\subseteq\sqrt{(0:_Rx)}$, hence there exists $n\in\N$ with $\ia^nx=0$, and therefore $\ia$ is half-centered.
\end{no}

\begin{prop}\label{d35}
If $\Gamma_\ia$ is a radical, then the following statements are equivalent: (i) $\ia$ is centered; (ii) $\ia$ is well-centered; (iii) Every $R$-module is weakly $\ia$-quasifair.
\end{prop}

\begin{proof}
``(i)$\Rightarrow$(ii)'': Suppose $\ia$ is centered. Let $M$ be an $R$-module with $\assf_R(M)\subseteq\var(\ia)$. Let $P$ be an $\ia$-torsionfree $R$-module, and let $f\in\hm{R}{M}{P}$. Then, \[\assf_R(M/\ke(f))\subseteq\supp_R(M/\ke(f))\subseteq\supp_R(M)\subseteq\var(\ia)\] by \ref{a20} G). Moreover, $f$ induces a monomorphism of $R$-modules $M/\ke(f)\rightarrowtail P$. So, $M/\ke(f)$ is $\ia$-torsionfree, and thus \[\assf_R(M/\ke(f))=\assf_R(M/\ke(f))\cap\var(\ia)=\emptyset\] by our hypothesis. It follows that $M/\ke(f)=0$ by \ref{a20} C), and therefore $f=0$. This shows that whenever $P$ is an $\ia$-torsionfree $R$-module, then $\hm{R}{M}{P}=0$. Since $\Gamma_\ia$ is a radical, this implies $M=\Gamma_\ia(M)$ by \ref{tor} E). So, $\ia$ is half-centered and therefore well-centered.

``(ii)$\Rightarrow$(iii)'': Suppose $\ia$ is well-centered. By \ref{a20} D) we have \[\assf_R(M)\cap\var(\ia)\subseteq\assf_R(\Gamma_\ia(M))\cup(\assf_R(M/\Gamma_\ia(M))\cap\var(\ia)).\] Our hypotheses imply $\assf_R(M/\Gamma_\ia(M))\cap\var(\ia)=\emptyset$, hence $\assf_R(M)\cap\var(\ia)\subseteq\assf_R(\Gamma_\ia(M))$, and then \ref{b10} b) yields the claim.

``(iii)$\Rightarrow$(i)'': Suppose every $R$-module is weakly $\ia$-quasifair. Let $M$ be an $\ia$-torsionfree $R$-module. Then, $\assf_R(M)\cap\var(\ia)=\assf_R(\Gamma_\ia(M))=\assf_R(0)=\emptyset$, and thus $\ia$ is centered.
\end{proof}

\begin{no}\label{d40}
Suppose that $\ass_R(M)=\assf_R(M)$ for every $R$-module $M$. Then, $\ia$ is centered. If in addition $\Gamma_\ia$ is a radical, and in particular if $\ia$ is of finite type, then $\ia$ is well-centered by \ref{d35}. It follows thus from \ref{a20} A) that ideals in noetherian rings are well-centered.
\end{no}


\section{Results for ideals of finite type}

\noindent\textit{Throughout this section, let $R$ be a ring, and let $\ia\subseteq R$ be an ideal.}\smallskip

\begin{prop}\label{c10}
If $\ia$ is principal and idempotent, then the $R$-module $R$ is $\ia$-fair.
\end{prop}

\begin{proof}
There exists $x\in R$ with $\ia=\langle x\rangle_R$, and $\Gamma_\ia(R)=(0:_Rx)$. We get \[R/\Gamma_\ia(R)=R/(0:_Rx)\cong\langle x\rangle_R=\ia\subseteq R,\] hence $\ass_R(R/\Gamma_\ia(R))\subseteq\ass_R(R)$ by \ref{a20} D), and so the claim follows from \ref{b10} c), \ref{tor} B) and \ref{d60}.
\end{proof}

\begin{cor}\label{c11}
If $R$ is absolutely flat and $\ia$ is of finite type, then the $R$-module $R$ is $\ia$-fair.
\end{cor}

\begin{proof}
Immediately from \ref{c10}, since ideals of finite type in absolutely flat rings are principal and idempotent.
\end{proof}

\begin{prop}\label{c20}
If $\ia$ is of finite type and $M$ is an $R$-module, then \[\assf_R(M/\Gamma_\ia(M))\cap\var(\ia)=\emptyset.\]
\end{prop}

\begin{proof}
We assume there exists $\ip\in\assf_R(M/\Gamma_\ia(M))\cap\var(\ia)$. By exactness of localisation we have a canonical isomorphism of $R_\ip$-modules $(M/\Gamma_\ia(M))_\ip\cong M_\ip/\Gamma_\ia(M)_\ip$. Since $\ia$ is of finite type, we have a canonical isomorphism of $R_\ip$-modules $M_\ip/\Gamma_\ia(M)_\ip\cong M_\ip/\Gamma_{\ia_\ip}(M_\ip)$ by \ref{tor} D), hence $\ip_\ip\in\assf_{R_\ip}((M/\Gamma_\ia(M))_\ip)=\assf_{R_\ip}(M_\ip/\Gamma_{\ia_\ip}(M_\ip))$ by \ref{a20} F). So, there exists $x\in(M_\ip/\Gamma_{\ia_\ip}(M_\ip))\setminus 0$ with $\ip_\ip\in\min(0:_{R_\ip}x)$. Since $\ip_\ip$ is the unique maximal ideal of $R_\ip$ it follows that $\ip_\ip=\sqrt{(0:_{R_\ip}x)}$. Since $\ip\in\var(\ia)$ we have $\ia_\ip\subseteq\ip_\ip=\sqrt{(0:_{R_\ip}x)}$. So, as $\ia$ is of finite type, the same holds for $\ia_\ip$, and thus there exists $n\in\N$ with $(\ia_\ip)^n\subseteq(0:_{R_\ip}x)$. This implies $x\in\Gamma_{\ia_\ip}(M_\ip/\Gamma_{\ia_\ip}(M_\ip))$. But as $\ia_\ip$ is of finite type we know that $\Gamma_{\ia_\ip}$ is a radical (\ref{tor} B)), and thus we get the contradiction $x=0$. Herewith the claim is proven.
\end{proof}

\begin{prop}\label{f20}
If $\ia$ is of finite type, then every $R$-module is weakly $\ia$-quasi\-fair.
\end{prop}

\begin{proof}
Let $M$ be an $R$-module with $\Gamma_\ia(M)=0$. Since $\ia$ is of finite type, \ref{c20} implies \[\assf_R(M/\Gamma_\ia(M))\cap\var(\ia)=\emptyset,\] and thus, by \ref{a20} D), it follows that \[\assf_R(M)\cap\var(\ia)\subseteq\assf_R(\Gamma_\ia(M))\cup(\assf_R(M/\Gamma_\ia(M))\cap\var(\ia))=\assf_R(\Gamma_\ia(M)).\] Therefore, $\ia$ is centered, and so \ref{tor} B) and \ref{d35} yield the claim.
\end{proof}

\begin{prop}\label{f40}
If $\ia$ is principal, then every $R$-module is weakly $\ia$-fair.
\end{prop}

\begin{proof}
Let $\ia=\langle a\rangle_R$ with $a\in R$, and let $M$ be an $R$-module. As $\Gamma_{\langle a\rangle_R}(M)$ is the kernel of the canonical morphism of $R$-modules $\eta_a\colon M\rightarrow M_a$, we have \[\assf_R(M/\Gamma_{\langle a\rangle_R}(M))=\{\ip\in\assf_R(M)\mid a\notin\ip\}\subseteq\assf_R(M)\] by \ref{a20} F), and thus \ref{c20} implies the claim.
\end{proof}

\begin{cor}
If $R$ is a Bezout ring and $\ia\subseteq R$ is an ideal of finite type, then every $R$-module is weakly $\ia$-fair.
\end{cor}

\begin{proof}
Immediately from \ref{f40}, since ideals of finite type of Bezout rings are principal.
\end{proof}

\begin{lemma}\label{f45}
Let $\ib\subseteq R$ be a weakly proregular ideal with $\Gamma_\ib\circ\Gamma_\ia=\Gamma_\ia$, and let $M$ be an $R$-module. Then, there is an isomorphism of $R$-modules \[(M/\Gamma_\ia(M))/\Gamma_\ib(M/\Gamma_\ia(M))\cong M/\Gamma_\ib(M).\]
\end{lemma}

\begin{proof}
Applying the functor $\Gamma_\ib$ to the short exact sequence \[0\longrightarrow\Gamma_\ia(M)\longrightarrow M\longrightarrow M/\Gamma_\ia(M)\longrightarrow 0\] and keeping in mind that $\Gamma_\ib\circ\Gamma_\ia=\Gamma_\ia$ we get an exact sequence \[0\longrightarrow\Gamma_\ia(M)\longrightarrow\Gamma_\ib(M)\longrightarrow\Gamma_\ib(M/\Gamma_\ia(M))\longrightarrow H_\ib^1(\Gamma_\ia(M)).\] By \ref{wr} B) we have $H_\ib^1\circ\Gamma_\ib=0$, hence by our hypothesis $H_\ib^1(\Gamma_\ia(M))=0$. The claim follows from this.\footnote{It is seen from the proof that in fact we do not need $\ib$ to be weakly proregular, but rather to have a finite generating family $\mathbf{b}$ such that the canonical morphism $\gamma^1_{\mathbf{b}}\colon H_\ib^1\rightarrow\check{H}^1_{\mathbf{b}}$ is an isomorphism.}
\end{proof}

\begin{prop}\label{f50}
Let $\ib,\ic\subseteq R$ be weakly proregular ideals such that $\ia=\ib+\ic$, and let $M$ be an $R$-module. If $M$ is weakly $\ib$-fair and weakly $\ic$-fair, then it is weakly $\ia$-fair.
\end{prop}

\begin{proof}
Let $\ip\in\assf_R(M/\Gamma_\ia(M))$. Since $\ia$ is of finite type this implies $\ip\notin\var(\ia)$ by \ref{c20}, and so we may without loss of generality suppose that $\ip\notin\var(\ib)$. Now, we have \[\assf_R(M/\Gamma_\ia(M))\subseteq\assf_R(\Gamma_\ib(M/\Gamma_\ia(M)))\cup\assf_R((M/\Gamma_\ia(M))/\Gamma_\ib(M/\Gamma_\ia(M)))\] and $\assf_R(\Gamma_\ib(M/\Gamma_\ia(M)))\subseteq\var(\ib)$ by \ref{a20} D) and \ref{b10} b). On use of \ref{f45} and our hypothesis it follows that \[\ip\in\assf_R((M/\Gamma_\ia(M))/\Gamma_\ib(M/\Gamma_\ia(M)))=\assf_R(M/\Gamma_\ib(M))\subseteq\assf_R(M).\] Thus, \ref{b10} d) yields the claim. 
\end{proof}

\begin{prop}\label{f60}
If there exists a generating family $(a_i)_{i=1}^n$ of $\ia$ such that $(a_i)_{i=1}^m$ is weakly proregular for every $m\in[1,n-1]$ and that $a_i$ is weakly proregular for every $i\in[2,n]$, then every $R$-module is weakly $\ia$-fair.
\end{prop}

\begin{proof}
We show the claim by induction on $n$. If $n\leq 1$ then we are done by \ref{b40} B) and \ref{f40}. Let $n>1$, and suppose the claim to hold for strictly smaller values of $n$. Let $\ib=\langle a_1,\ldots,a_{n-1}\rangle_R$ and $\ic=\langle a_n\rangle_R$. Then, $\ib$ and $\ic$ are weakly proregular, and every $R$-module is weakly $\ib$-fair and weakly $\ic$-fair by our hypothesis. Therefore, \ref{f50} implies that every $R$-module is weakly $\ia$-fair, and thus the claim is proven.
\end{proof}

\begin{prop}\label{g40}
Let $\ia$ be of finite type, and let $M$ be an $R$-module such that every (weakly) associated prime of $M/\Gamma_\ia(M)$ is of finite type. Then, $M$ is (weakly) $\ia$-fair.
\end{prop}

\begin{proof}
By \ref{a20} A) it suffices to show the claim about assassins. By \ref{tor} F) we have a monomorphism $M/\Gamma_\ia(M)\rightarrowtail\ilim_{n\in\N}\hm{R}{\ia^n}{M}$, and so \ref{a20} D), E) and H) imply \[\ass_R(M/\Gamma_\ia(M))\subseteq\{\ip\in\ass_R(\ilim_{n\in\N}\hm{R}{\ia^n}{M})\mid\ip\text{ is of finite type}\}\subseteq\]\[\bigcup_{n\in\N}\ass_R(\hm{R}{\ia^n}{M})\subseteq\ass_R(M).\] The claim follows now from \ref{c20}.
\end{proof}

\begin{qus}
We are left with the following questions, on whose answers we dare not speculate.
\begin{aufz}
\item[$(*)$] Suppose $\Gamma_\ia$ is a radical. Is then every $R$-module weakly $\ia$-quasifair, $\ia$-fair or weakly $\ia$-fair?
\item[$(**)$] Suppose $\ia$ is of finite type. Is then every $R$-module $\ia$-fair or weakly $\ia$-fair?
\end{aufz}
\end{qus}


\textbf{Acknowledgements.} I thank the referee for his careful reading and his suggestions, and -- as so often -- Pham Hung Quy for suggesting nice counterexamples.


\end{document}